\documentclass[11pt]{article}

\usepackage{amsmath, amsthm, amssymb, enumerate}
\usepackage{fullpage}
\usepackage{tikz}
\usetikzlibrary{decorations,decorations.pathmorphing}

\date{}
\title{\vspace{-0.8cm}Hamiltonicity, independence number, and pancyclicity}
\author{
Choongbum Lee \thanks{Department of Mathematics, UCLA, Los
Angeles, CA, 90095. Email: choongbum.lee@gmail.com. Research
supported in part by a Samsung Scholarship.}
\and
Benny Sudakov \thanks{Department of Mathematics, UCLA, Los Angeles, CA 90095.
Email: bsudakov@math.ucla.edu. Research supported
in part by NSF grant DMS-1101185, NSF CAREER award DMS-0812005 and by USA-Israeli BSF grant.}
}

\theoremstyle{plain}
\newtheorem{THM}{Theorem}[section]
\newtheorem{DEF}[THM]{Definition}
\newtheorem{PROP}[THM]{Proposition}
\newtheorem{LEMMA}[THM]{Lemma}

\theoremstyle{definition}

\newenvironment{pfof}[1]{\noindent{\bf Proof of #1.}}{\qed\medskip}
\newenvironment{REM}{\noindent\textbf{Remark.}}{\medskip}

\renewcommand{\deg}{\textrm{d}}

\begin{document}
\maketitle

\begin{abstract}
A graph on $n$ vertices is called pancyclic if it contains a cycle
of length $\ell$ for all $3 \le \ell \le n$. In 1972, Erd\H{o}s proved
that if $G$ is a Hamiltonian graph
on $n > 4k^4$ vertices with independence number $k$, then $G$
is pancyclic. He then suggested that $n = \Omega(k^2)$ should already be
enough to guarantee pancyclicity.
Improving on his and some other later results, we prove that there
exists a constant $c$ such that $n > ck^{7/3}$ suffices.
\end{abstract}

%%
%%
%%
%%
%%
%%  INTRODUCTION
%%
%%
%%
%%
%%
\section{Introduction}

A {\em Hamilton cycle} of a graph is a cycle which passes through
every vertex of the graph exactly once, and a graph is called
{\em Hamiltonian} if it contains a Hamilton cycle. Determining whether
a given graph is Hamiltonian is one of the central questions in graph
theory, and there are numerous results
which establish sufficient conditions for Hamiltonicity.
For example, a celebrated result of Dirac asserts that every graph of
minimum degree at least $\lceil n/2 \rceil$ is Hamiltonian.
A graph is {\em pancyclic} if it contains a cycle of length $\ell$ for
all $3 \le \ell \le n$. By definition, every pancyclic graph is Hamiltonian,
but it is easy to see that the converse is not true. Nevertheless,
these two concepts are closely related and many nontrivial conditions which imply
Hamiltonicity also imply pancyclicity of a graph. For instance,
extending Dirac's Theorem,
Bondy \cite{Bondy1} proved that every graph of minimum degree at least
$\left\lceil n/2 \right\rceil$ either is the complete bipartite $K_{\lceil n/2 \rceil, \lfloor n/2 \rfloor}$, or is pancyclic. Moreover, in \cite{Bondy2}, he
made a meta conjecture in this context which says that almost any
non-trivial condition on a graph which implies that the graph
is Hamiltonian also implies that the graph is pancyclic
(there may be a simple family of exceptional graphs).

Let the {\em independence number} $\alpha(G)$ of a graph $G$ be the
order of a maximum independent set of $G$. A classical result of
Chv\'{a}tal and Erd\H{o}s \cite{ChEr} says that every graph $G$
whose vertex connectivity (denoted $\kappa(G)$) is at least as large
as its independence number is Hamiltonian. Motivated by Bondy's
metaconjecture, Amar, Fournier, and Germa \cite{AmFoGe} obtained
several results on the lengths of cycles in a graph $G$ that
satisfies the Chv\'atal-Erd\H{o}s condition $\kappa(G) \ge
\alpha(G)$, and conjectured that if such a graph $G$ is not
bipartite then either $G=C_5$, or $G$ contains cycles of length
$\ell$ for all $4 \le \ell \le n$ (Lou \cite{Lou} made some partial
progress towards this conjecture). In a similar context, Jackson and
Ordaz \cite{JaOr} conjectured that that every graph $G$ with
$\kappa(G) > \alpha(G)$ is pancyclic. Keevash and Sudakov
\cite{KeSu} proved that there exists an absolute constant $c$ such
that $\kappa(G) \ge c \alpha(G)$ is sufficient for pancyclicity.

In this paper, we study a relation between Hamiltonicity,
pancyclicity, and the independence number of a graph. Such relation
was first studied by Erd\H{o}s \cite{Erdos}. In 1972, he proved a
conjecture of Zarins by establishing the fact that every Hamiltonian
graph $G$ on $n \ge 4k^4$ vertices with $\alpha(G) \le k$ is
pancyclic (see also \cite{FlLiMaWo} for a different proof of a
weaker bound). Erd\H{o}s also suggested that the bound $4k^4$ on the
number of vertices is probably not tight, and that the correct order
of magnitude should be $\Omega(k^2)$. The following graph shows that
this, if true, is indeed best possible. Let $K_1, \cdots, K_k$ be
disjoint cliques of size $k-2$, where each $K_i$ has two
distinguished vertices $v_i$ and $w_i$. Let $G$ be the graph
obtained by connecting $v_i \in K_i$ and $w _{i+1} \in K_{i+1}$ by
an edge (here addition is modulo $k$). One can easily show that this
graph is Hamiltonian, has $k(k-2)$ vertices, and independence number
$k$. However, this graph does not contain a cycle of length $k-1$
(thus is not pancyclic), since every cycle either is a subgraph of
one of the cliques, or contains at least one vertex from each clique
$K_i$. The former type of cycles have length at most $k-2$, and the
later type of cycles have length at least $2k$.

Recently, Keevash and Sudakov \cite{KeSu} improved Erd\H{o}s' result
and showed that $n>150k^3$ already implies pancyclicity. Our main
theorem further improves this bound.

\begin{THM} \label{thm_main2}
There exists a constant $c$ such that for every positive integer
$k$, every Hamiltonian graph on $n \ge ck^{7/3}$ vertices with
$\alpha(G) \le k$ is pancyclic.
\end{THM}

Suppose that one established the fact that for some function $f(k)$,
every Hamiltonian graph on $n \ge f(k)$ vertices with $\alpha(G) \le
k$ contains a cycle of length $n-1$. Then, by iteratively applying
this result, one can easily see that for every constant $C \ge 1$,
every graph on $n \ge Cf(k)$ vertices with $\alpha(G) \le k$
contains cycles of all length between $\frac{n}{C}$ and $n$. This
simple observation were used in both \cite{Erdos} and \cite{KeSu},
where they found cycles of length linear in $n$ using this method,
and then found cycles of smaller lengths using other methods. Thus
the problem finding a cycle of length $n-1$ is a key step in proving
pancyclicity. Keevash and Sudakov suggested that if one just is
interested in this problem, then the bound between the number of
vertices and independence number can be significantly improved. More
precisely, they asked whether there is an absolute constant $c$ such
that every Hamiltonian graph on $n \ge ck$ vertices with
independence number $k$ contains a cycle of length $n-1$.

Despite the fact that this bound suggested by Keevash and Sudakov is
only linear in $k$, even improving Erd\H{o}s' original estimate of
$n = \Omega(k^3)$ was not an easy task. Moreover, as we will explain
in the concluding remarks, currently the bottleneck of proving
pancyclicity lies in this step of finding a cycle of length $n-1$.
Thus in order to prove Theorem~\ref{thm_main2}, we partially answer
Keevash and Sudakov's question for the range $n \ge ck^{7/3}$, and
combine this result with tools developed in \cite{KeSu}. Therefore,
the main focus of our paper will be to prove the following theorem.

\begin{THM} \label{thm_main}
There exists a constant $c$ such that for every positive integer $k$,
every Hamiltonian
graph on $n \ge ck^{7/3}$ vertices with $\alpha(G) \le k$
contains a cycle of length $n-1$.
\end{THM}

In Section \ref{section_thm2}, we state a slightly stronger form of
Theorem \ref{thm_main}, and use it to deduce
Theorem \ref{thm_main2}. Then in Sections \ref{section_structural}
and \ref{section_thm1}, we prove the strengthened version Theorem \ref{thm_main}.
To simplify the presentation, we often omit floor and ceiling signs whenever
these are not crucial and make no attempts to optimize absolute constants involved.

%%
%%
%%
%%
%%
%%  pancyclicity
%%
%%
%%
%%
%%
\section{Pancyclicity}
\label{section_thm2}

%In this section, we prove Theorem \ref{thm_main2} from Theorem \ref{thm_main}.
In order to prove Theorem \ref{thm_main2}, we use the following
slightly stronger form of Theorem \ref{thm_main} whose proof
will be given in the next two sections.

\begin{THM} \label{thm_mainstrong}
There exists a constant $c$ such that for every positive integer $k$,
every Hamiltonian graph on $n \ge ck^{7/3}$ with $\alpha(G) \le k$
contains a cycle of length $n-1$. Moreover, for an arbitrary fixed
set of vertices $W$ of size $|W| \le 20k^2$, we can find such a cycle
which contains all the vertices of $W$.
\end{THM}

As mentioned in the Introduction, Theorem \ref{thm_mainstrong} will be
used to find cycles of linear lengths.
The following two results from \cite[Theorem 1.3 and Lemma 3.2]{KeSu} allows us to
find cycle lengths in the range not covered by Theorem \ref{thm_mainstrong}.
\begin{THM} \label{thm_smallcycle}
If $G$ is a graph with $\delta(G) \ge 300 \alpha(G)$ then $G$ contains a cycle of
length $\ell$ for all $3 \le \ell \le \delta(G)/81$.
\end{THM}

\begin{LEMMA} \label{lemma_medcycle}
Suppose $G$ is a graph with independence number $\alpha(G) \le k$ and $V(G)$ is
partitioned into two parts $A$ and $B$ such that
\begin{enumerate}[(i)]
  \setlength{\itemsep}{1pt}
  \setlength{\parskip}{0pt}
  \setlength{\parsep}{0pt}
\item $G[A]$ is Hamiltonian.
\item $|B| \ge 9k^2 + k + 1$, and
\item every vertex in $B$ has at least 2 neighbors in $A$.
\end{enumerate}
Then $G$ contains a cycle of length $\ell$ for all $2k+1 + \left\lfloor \log_2(2k+1) \right\rfloor \le \ell \le |A|/2$.
\end{LEMMA}

\begin{pfof}{Theorem \ref{thm_main2}}
Note that the conclusion is immediate if $k=1$. Thus we may assume
that $k \ge 2$. Let $c$ be the maximum of the constant coming from
Theorem \ref{thm_mainstrong} and 300 and $G$ be Hamiltonian graph on
$n=3ck^{7/3}$ vertices such that $\alpha(G) \le k$. By repeatedly
applying Theorem \ref{thm_mainstrong} with $W=\emptyset$, we can
find cycles of length $ck^{7/3}$ to $3ck^{7/3}$.

Moreover, as we will see,
by carefully using Theorem \ref{thm_mainstrong} in the previous step,
we can prepare a setup for applying Lemma \ref{lemma_medcycle}.
Let $C_1$ be the cycle of length $n-1$ obtained by Theorem \ref{thm_mainstrong},
and let $v_1$ be the vertex not contained in $C_1$.
We know that $v_1$ has at least 2 neighbors in $C_1$. Let $W_1$ be arbitrary two
vertices out of them. By applying Theorem \ref{thm_mainstrong} with $W=W_1$,
we can find a cycle $C_2$ of length $n-2$ which contains $W_1$. Let $v_2$
be the vertex contained in $C_1$ but not in $C_2$, and let $W_2$ be the
union of $W_1$ and arbitrary two neighbors of $v_2$ in $C_2$. We can repeat it
$10k^2$ times (note that we maintain $|W| \le 20k^2$),
to obtain a cycle $C_{10k^2}$ of length $n-10k^2$, and vertices $v_1, \cdots, v_{10k^2}$
so that each $v_i$ has at least 2 neighbors in the cycle $C_{10k^2}$. Since
$10k^2 \ge 9k^2 + k + 1$, by Lemma \ref{lemma_medcycle}, $G$ contains a cycle of
length $\ell$ for all $2k+1 + \left\lfloor \log_2(2k+1) \right\rfloor \le \ell \le (n-10k^2)/2$.

Now we find all the remaining cycle lengths. From the graph $G$,
pick one by one, a vertex of degree less than $ck^{4/3}$, and
remove it together with its neighbors. Note
that since the picked vertices form an independent set in $G$, at
most $k$ vertices will be removed. Therefore, when there are no more
vertices to pick, at least $3ck^{7/3} - k\cdot (ck^{4/3}+1) >
ck^{7/3}$ vertices remain, and the induced subgraph of $G$ on these
vertices will be of minimum degree at least $ck^{4/3}$. Since
$ck^{4/3} \ge 300 k \ge 300 \alpha(G)$, by Theorem
\ref{thm_smallcycle}, $G$ contains a cycle of length $\ell$ for all
$3 \le \ell \le (c/81)k^{4/3}$.

By noticing the inequalities $(n-10k^2)/2 = (3ck^{7/3} - 10k^2)/2 \ge ck^{7/3}$ and $(c/81)k^{4/3} \ge 2k+1 + \left\lfloor \log_2(2k+1) \right\rfloor$ we can see the existence of cycles of all possible lengths.
\end{pfof}

%%
%%
%%
%%
%%
%%  Structural lemma
%%
%%
%%
%%
%%
\section{A structural lemma}
\label{section_structural}

In Sections \ref{section_structural} and \ref{section_thm1},
we will prove Theorem \ref{thm_mainstrong}. Given a Hamiltonian graph
on $n$ vertices, one can easily see that there are many ways one
can find a cycle of length $n-1$, if certain `chords' are
present in the graph. Our strategy is to find such chords that are `nicely' arranged.
In particular, in this section, we consider pairs of chords and the way they cross
each other in order to deduce some
structure of our graph. Then in the next section, we prove the
main theorem by considering certain triples of chords, which we call
semi-triangles.

Throughout this section, let $G$ be a fixed graph on $n \ge 80k^2$
vertices such that $\alpha(G) \le k$, and let $W$ be a fixed set of
vertices such that $|W| \le 20k^2$. Note that the bound on the
number of vertices is weaker than that of Theorem
\ref{thm_mainstrong}. The results developed in this section still
holds under this weaker bound, and we only need the stronger bound
$n \ge ck^{7/3}$ in the next section.
Since our goal is to prove the existence of a cycle of length $n-1$,
assume to the contrary that $G$ does not contain a cycle of length
$n-1$. Under these assumptions, we will prove a structural lemma on
the graph $G$ which will immediately imply a slightly weaker form of
Theorem \ref{thm_mainstrong} where the bound on the number of
vertices is replaced by $\Omega(k^{5/2})$. In the next section, we
will apply this structural lemma more carefully to prove Theorem
\ref{thm_mainstrong}.

One of the main ingredients of the proof is the following
proposition proved by Erd\H{o}s \cite{Erdos}, whose idea has its
origin in \cite{ChEr}.

\begin{PROP} \label{prop_extension}
For all $1 \le i \le 12k$, $G$ does not have a cycle of length $n-i$
containing $W$ for which all the vertices not in this cycle have
degree at least $13k$.
\end{PROP}
\begin{proof}
Assume that $C$ is the vertex set of
a cycle given as above and let $X = V(G) \setminus C$.
We will show that there exists a cycle of length $|C|+1$ which contains $C$.
By repeatedly applying the same argument, we can show the existence
of a cycle of length $n-1$. Since this contradicts our hypothesis,
we can conclude that $G$ cannot contain a cycle as above.

Consider a vertex $x \in X$. Since $|X| \le 12k$ and $\deg(x) \ge
13k$, we know that the number of neighbors of $x$ in $C$ is at least
$k$. Without loss of generality, let $C = \{ 1,2,\cdots,n-i \}$, and
assume that the vertices are labeled in the order in which
they appear on the cycle. Let $w_1,
\cdots, w_k$ be distinct neighbors of $x$ in $C$. Then since $G$ has
independence number less than $k$, there exists two vertices $w_i -
1, w_j - 1$ which are adjacent (subtraction is modulo $n-i$). Then
$G$ contains a cycle $x, w_i, w_i + 1, \cdots, w_j - 1, w_i - 1, w_i
- 2, \cdots, w_j, x$ of length $n-i+1$.
\end{proof}

\noindent In view of Proposition \ref{prop_extension}, we make the
following definition.

\begin{DEF}
Let a {\em contradicting cycle} be a cycle containing $W$, of length
$n-i$ for some $1 \le i \le 12k$, for which all the vertices not in
this cycle have degree at least $13k$.
\end{DEF}

Thus Proposition \ref{prop_extension} is equivalent to saying that
$G$ does not contain a contradicting cycle (under the assumption that
$G$ does not contain a cycle of length $n-1$). By considering several
cases, we will show that there always exists a contradicting cycle,
from which we can deduce a contradiction on our assumption that
there is no cycle of length $n-1$ in $G$. The next simple
proposition will provide a set-up for this argument.

\begin{PROP} \label{prop_fewsmalldegree}
$G$ contains at most $13k^2$ vertices of degree less than $13k$.
\end{PROP}
\begin{proof}
Assume that there exists a set $U$ of at least $13k^2+1$ vertices of
degree less than $13k$, and let $G' \subset G$ be the subgraph of $G$
induced by $U$. Take a vertex of $G'$ of degree less than $13k$,
remove it and all its neighbors from $G'$, and repeat the process.
This produces an independent set of size at least $\lceil
(13k^2+1) / 13k \rceil = k+1$ which is a contradiction.
\end{proof}

Assume that we are given a Hamilton cycle of $G$. Place the vertices
of $G$ on a circle in the plane according to the order they appear
in the Hamilton cycle and label the vertices by elements in $[n]$
accordingly. Consider the $40k^2$ intervals $[1+(i-1)\lfloor
\frac{n}{40k^2} \rfloor, i \lfloor \frac{n}{40k^2} \rfloor]$ for $i=1, \cdots,
40k^2$ consisting of consecutive vertices on the cycle. Take the
intervals which only consist of vertices not in $W$ of degree at
least $13k$. Let $t$ be the number of such intervals and let $I_1,
I_2, \cdots, I_t$ be these intervals (see Figure
\ref{fig:figstructural}). By Proposition \ref{prop_fewsmalldegree},
the number of intervals which contain a vertex from $W$ or of degree
less than $13k$ is at most $|W| + 13k^2$, and therefore
\[ t \ge 40k^2 - 13k^2 - |W| \ge 7k^2. \]

\begin{figure}[t]
  \centering
  \begin{tabular}{ccc}
    %%
%% Intervals
%%

\begin{tikzpicture}
  \foreach \a / \r / \j in {
    30/29/I_3, 
    50/29/I_2,
    70/29/I_1, 
    120/29/I_t, 
    -30/29/I_4
  }
  {
    \draw (\a-10:24mm) -- (\a-10:26mm);
    \draw (\a+10:24mm) -- (\a+10:26mm);
    \draw (\a:\r mm) node {{$\j$}};
    \draw (\a-10:26mm) arc (\a-10:\a+10:26mm);
  }

  \draw [thick] (0:25mm) arc (0:360:25mm);

\end{tikzpicture} & \hspace{0.5in} & %%
%% Color of reduced graph
%%

\begin{tikzpicture}
% Circle  
  \draw [thick] (0:25mm) arc (0:360:25mm);

% Intervals
  \foreach \a / \b / \t / \c in {
    20/140/Red/20, 200/320/Blue/20
  }
  {
    \draw (\a-15:24mm) -- (\a-15:26mm);
    \draw (\a+15:24mm) -- (\a+15:26mm);
    \draw (\a-15:26mm) arc (\a-15:\a+15:26mm);

    \draw (\b-15:24mm) -- (\b-15:26mm);
    \draw (\b+15:24mm) -- (\b+15:26mm);
    \draw (\b-15:26mm) arc (\b-15:\b+15:26mm);

    \draw (\a+60:\c mm) node {{\t}};    
  }

%  Red
		\draw [fill=black] (10:25mm) circle (0.5mm);
		\draw [fill=black] (30:25mm) circle (0.5mm);
    \draw (10:29 mm) node {$x_2$};    
    \draw (30:29 mm) node {$x_1$};    
		
		\draw [fill=black] (130:25mm) circle (0.5mm);
		\draw [fill=black] (150:25mm) circle (0.5mm);
    \draw (130:29 mm) node {$y_2$};    
    \draw (150:29 mm) node {$y_1$};    

    \draw (10:25mm) -- (130:25mm);
    \draw (30:25mm) -- (150:25mm);

%  Blue
		\draw [fill=black] (190:25mm) circle (0.5mm);
		\draw [fill=black] (210:25mm) circle (0.5mm);
    \draw (190:29 mm) node {$y_2$};    
    \draw (210:29 mm) node {$y_1$};    
		
		\draw [fill=black] (310:25mm) circle (0.5mm);
		\draw [fill=black] (330:25mm) circle (0.5mm);
    \draw (310:29 mm) node {$x_2$};    
    \draw (330:29 mm) node {$x_1$};    

    \draw (190:25mm) -- (330:25mm);
    \draw (210:25mm) -- (310:25mm);

\end{tikzpicture}
  \end{tabular}
  \caption{Intervals $I_i$, and the type of edges in the graph $H$.}
  \label{fig:figstructural}
\end{figure}
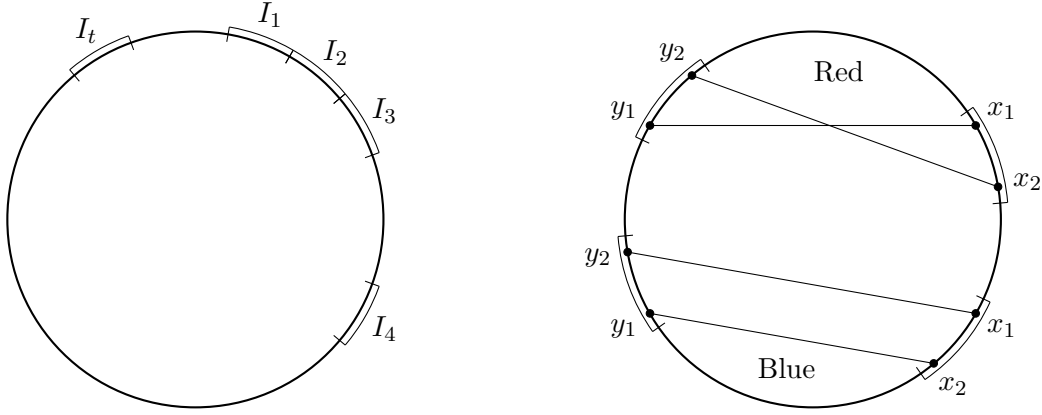

For each interval $I_j$, let $I_j'$ be the
set of first at most $k+1$ odd vertices in it (thus $I_j'$
is the set of all odd vertices in $I_j$ if $|I_j| \le 2(k+1)$).
If there exists an edge inside $I_j'$ then since
$I_j'$ lies in an interval of length at most $2k+2$, we can find a
contradicting cycle. Therefore $I_j'$ is an independent set
of size at least $\min\{k+1, \lfloor \frac{n}{80k^2} \rfloor \}$. However,
since the independence number of the graph is at most $k$, the
first case $|I_j'|=k+1$ gives us a contradiction. Therefore, we
may assume that $|I_j'| \le k$, and thus $I_j'$ lies in an interval
of length at most $2k$.

Consider an auxiliary graph $H$ on the vertex set $[t]$ so that
$i,j$ are adjacent if and only if there exists an edge between
$I_i'$ and $I_j'$. Furthermore, color the edges of $H$ into three
colors according to the following rule (see Figure
\ref{fig:figstructural}).
\begin{enumerate}[(i)]
  \setlength{\itemsep}{1pt} \setlength{\parskip}{0pt}
  \setlength{\parsep}{0pt}
\item Red if there exists $x_1, x_2 \in I_i', y_1, y_2 \in I_j'$ such that $x_1 < x_2$, $y_1 < y_2$ and $x_1$ is adjacent to $y_1$, and $x_2$ is adjacent to $y_2$.
\item Blue if not colored red, and there exists $x_1, x_2 \in I_i', y_1, y_2 \in I_j'$ such that $x_1 < x_2$, $y_1 < y_2$ and $x_1$ is adjacent to $y_2$, and $x_2$ is adjacent to $y_1$.
\item Green if not colored red nor blue.
\end{enumerate}
A red edge in the graph $H$ will give a cycle $x_1-y_1-x_2-y_2-x_1$,
see Figure \ref{fig:figstructural2}. The length of the cycle is at
least $n- 4k$ since each $I_i'$ lies in an interval of length at
most $2k$, and is at most $n-2$ since there always exist vertices
between $x_1, x_2$ and between $y_1, y_2$. Moreover, the cycle
contains the set $W$ since $W$ does not intersect the intervals
$I_i$. Therefore it is a contradicting cycle. Thus we may assume
that there does not exist red edges in $H$.

Consider the following drawing of the subgraph of $H$ induced by the
blue edges. First place all the vertices of the graph $G$ on the
cycle along the given order. A vertex of $H$, which corresponds to
an interval $I_i$, will be placed on the circle in the middle of the
interval $I_i$. Draw a straight line between $I_i$ and $I_j$ if
there is a blue edge. Assume that there exists a crossing in this
drawing. Then this gives a situation as in Figure
\ref{fig:figstructural2} which gives the cycle
$x_1-y_2-x_3-y_4-y_1-x_2-y_3-x_4-x_1$. This cycle has length at
least $n - 4\cdot 2k \ge n - 8k$ and at most $n - 4$, hence is a
contradicting cycle.

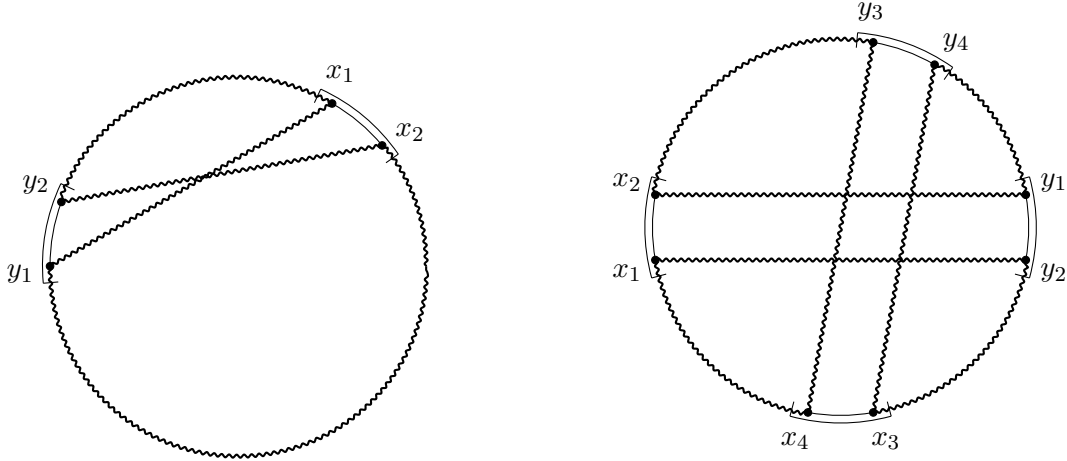
\begin{figure}[t]
  \centering
  \begin{tabular}{ccc}
    %%
%% Red cycle
%%

\begin{tikzpicture}
% Intervals
  \foreach \a/\j/\k in {
    50/x_2/x_1, 170/y_2/y_1
  }
  {
    \draw (\a-15:24mm) -- (\a-15:26mm);
    \draw (\a-12:29mm) node {{$\j$}};
    \draw (\a+15:24mm) -- (\a+15:26mm);
    \draw (\a+12:29mm) node {{$\k$}};
    \draw (\a-15:26mm) arc (\a-15:\a+15:26mm);
  }

%  Red
        \draw [fill=black] (40:25mm) circle (0.5mm);
        \draw [fill=black] (60:25mm) circle (0.5mm);

        \draw [fill=black] (160:25mm) circle (0.5mm);
        \draw [fill=black] (180:25mm) circle (0.5mm);

    \draw [thick, decorate, decoration={snake, amplitude=.2mm, segment length=1mm}] (40:25mm) -- (160:25mm);
    \draw [thick, decorate, decoration={snake, amplitude=.2mm, segment length=1mm}] (60:25mm) -- (180:25mm);

% Circle
  \draw (40:25mm) arc (40:60:25mm);
  \draw (160:25mm) arc (160:180:25mm);
  \draw [thick, decorate, decoration={snake, amplitude=.2mm, segment length=1mm}] (180:25mm) arc (180:360:25mm);
  \draw [thick, decorate, decoration={snake, amplitude=.2mm, segment length=1mm}] (0:25mm) arc (0:40:25mm);
  \draw [thick, decorate, decoration={snake, amplitude=.2mm, segment length=1mm}] (60:25mm) arc (60:160:25mm);

\end{tikzpicture} & \hspace{0.5in} & %%
%% Blue cycle
%%

\begin{tikzpicture}
% Intervals
  \foreach \a/\j/\k in {
    0/y_2/y_1, 70/y_4/y_3, 180/x_2/x_1, 270/x_4/x_3
  }
  {
    \draw (\a-15:24mm) -- (\a-15:26mm);
    \draw (\a-12:29mm) node {{$\j$}};
    \draw (\a+15:24mm) -- (\a+15:26mm);
    \draw (\a+12:29mm) node {{$\k$}};
    \draw (\a-15:26mm) arc (\a-15:\a+15:26mm);
  }

%  Two blue
        \draw [fill=black] (-10:25mm) circle (0.5mm);
        \draw [fill=black] (10:25mm) circle (0.5mm);

        \draw [fill=black] (60:25mm) circle (0.5mm);
        \draw [fill=black] (80:25mm) circle (0.5mm);

        \draw [fill=black] (170:25mm) circle (0.5mm);
        \draw [fill=black] (190:25mm) circle (0.5mm);

        \draw [fill=black] (260:25mm) circle (0.5mm);
        \draw [fill=black] (280:25mm) circle (0.5mm);

    \draw [thick, decorate, decoration={snake, amplitude=.2mm, segment length=1mm}] (-10:25mm) -- (190:25mm);
    \draw [thick, decorate, decoration={snake, amplitude=.2mm, segment length=1mm}] (10:25mm) -- (170:25mm);
    \draw [thick, decorate, decoration={snake, amplitude=.2mm, segment length=1mm}] (60:25mm) -- (280:25mm);
    \draw [thick, decorate, decoration={snake, amplitude=.2mm, segment length=1mm}] (80:25mm) -- (260:25mm);

% Circle
  \draw (-10:25mm) arc (-10:10:25mm);
  \draw (60:25mm) arc (60:80:25mm);
  \draw (170:25mm) arc (170:190:25mm);
  \draw (260:25mm) arc (260:280:25mm);

  \draw [thick, decorate, decoration={snake, amplitude=.2mm, segment length=1mm}] (10:25mm) arc (10:60:25mm);
  \draw [thick, decorate, decoration={snake, amplitude=.2mm, segment length=1mm}] (80:25mm) arc (80:170:25mm);
  \draw [thick, decorate, decoration={snake, amplitude=.2mm, segment length=1mm}] (190:25mm) arc (190:260:25mm);
  \draw [thick, decorate, decoration={snake, amplitude=.2mm, segment length=1mm}] (280:25mm) arc (280:350:25mm);

\end{tikzpicture}
  \end{tabular}
  \caption{Contradicting cycles for a red edge, and two crossing blue edges in $H$.}
  \label{fig:figstructural2}
\end{figure}

Therefore, the subgraph of $H$ induced by blue edges form a planar
graph. This implies that there exists a subset of $[t]$ of size at
least $t/5$ which does not contain any blue edge (note that here we
use the fact that every planar graph is 5-colorable). By slightly
abusing notation, we will only consider these intervals, and relabel
the intervals as $I_1, \cdots, I_s$ where $s \ge t/5 > k^2$.

\begin{LEMMA} \label{lemma_structural}
Let $a_1, \cdots, a_p \in [s]$ be distinct integers and let $X_{a_i}
\subset I_{a_i}'$ for all $i$. Then $X_{a_1} \cup \cdots \cup
X_{a_p}$ contains an independent set of size at least
\[ \sum_{i=1}^{p} |X_{a_i}|  - \frac{p(p-1)}{2}.  \]
\end{LEMMA}
\begin{proof}
The proof of this lemma relies on a fact about green edges in the
auxiliary graph $H$. Assume that there exists a green edge between
$i$ and $j$ in $H$. Then by the definition, since the edge is
neither red nor blue, we know that there is no matching in $G$ of
size 2 between $I_i'$ and $I_j'$. Therefore there exists a vertex
$v$ which covers all the edges between $I_i'$ and $I_j'$.

Now consider the following process of constructing an independent
set $J$. Take $J = \emptyset$ in the beginning. At step $i$, add
$X_{a_i}$ to the set $J$. By the previous observation, for $j < i$,
all the edges between $X_{a_i}$ and $X_{a_j}$ can be deleted by
removing at most one vertex (either from $X_{a_i}$ or $X_{a_j}$).
Therefore $J \cup X_{a_i}$ can be made into an independent set by
removing at most $i-1$ vertices. By iterating the process, we can
obtain an independent set of size at least $\sum_{i=1}^{p}
\left(|X_{a_i}| - (i-1) \right) \ge \sum_{i=1}^{p} |X_{a_i}|  -
\frac{p(p-1)}{2}$.
\end{proof}

\begin{REM} As mentioned before, this lemma already implies a weaker
version of Theorem \ref{thm_mainstrong} where the bound is replaced
by $n = 240k^{5/2}$. To see this, assume that we have a graph on at
least $240k^{5/2}$ vertices. Take $X_i = I_i'$ for $i=1, \cdots,
\lceil k^{1/2} \rceil$ in this lemma and notice that $|I_i'| \ge
\min \{ k+1, \lfloor 3k^{1/2} \rfloor\}$. As we have seen before,
$|I_i'| = k+1$ cannot happen. On the other hand, $|I_i'|=\lfloor
3k^{1/2} \rfloor$ implies the existence of an independent set of
size at least
\[ \lfloor 3k^{1/2} \rfloor \cdot \lceil k^{1/2} \rceil- \frac{\lceil k^{1/2} \rceil (\lceil k^{1/2} \rceil -1)}{2} \ge (3k^{1/2}-1)k^{1/2} - \frac{k+k^{1/2}}{2} > k, \]
which gives a contradiction.
\end{REM}

%% Everything in this section holds in the range $n=\Omega(k^2)$.

%%
%%
%%
%%
%%
%%  main theorem
%%
%%
%%
%%
%%
\section{Proof of Theorem \ref{thm_mainstrong}}
\label{section_thm1}

In this section, we will prove Theorem \ref{thm_mainstrong} which
says that there exists a constant $c$ such that every Hamiltonian
graph on $n \ge ck^{7/3}$ with $\alpha(G) \le k$
contains a cycle of length $n-1$. We will first focus on
proving the following relaxed statement: there exists
$k_0$ such that for $k \ge k_0$, every Hamiltonian graph
on $n \ge 960 k^{7/3}$ vertices with $\alpha(G) \le k$ contains
a cycle of length $n-1$. Note that for the range $k < k_0$, since there
exists a constant $c'$ such that $c'k^{7/3} \ge 240k^{5/2}$, by the
remark at the end of the previous section, the bound $n \ge
c'k^{7/3}$ will imply pancyclicity. Therefore by taking $\max\{960,
c'\}$ as our final constant, the result we prove in
this section will in fact imply Theorem \ref{thm_mainstrong}.
By relaxing the statement as above, we may assume that $k$ is large
enough. This will simplify many calculations. In
particular, it allows us to ignore the floor and ceiling signs in
this section.

\medskip

Now we prove the above relaxed statement using the tools
we developed in the previous section. Assume that $n \ge 960 k^{7/3}$
and $k$ is large enough.
Recall that we have independent sets $I_1', \cdots, I_s'$ such that $s >
k^2$ and $|I_i'| \ge \lfloor \frac{n}{80k^2} \rfloor \ge 12k^{1/3}$ for all $i$. For
each $i$, let $M_i$ and $L_i$ be the smaller $|I_i'|/2$ vertices and
larger $|I_i'|/2$ vertices of $I_i'$ in the cycle order given in the
previous section, and call them as the {\em main set} and {\em leftover
set}, respectively. Note that $M_i$ and $L_i$ both have size at
least $6k^{1/3}$. For a vertex $v$, call a set $M_j$ (or an index $j$) as a
{\em neighboring main set} of $v$ if $v$ contains a neighbor in
$M_j$.

\begin{LEMMA} \label{lemma_mindegreereducedgraph}
There exists a subcollection of indices $S \subset [s]$ such that
the following holds. For every $i \in S$, the set $M_i$ contains at
least $3k^{1/3}$ vertices which each have at least $k$
neighboring main sets whose indices lie in $S$.
\end{LEMMA}
\begin{proof}
In order to find the set of indices $S$ described in the statement,
consider the process of removing the main sets which do not satisfy
the condition one by one. If the process ends before we run out of sets,
then the remaining indices will satisfy the condition.

Let $J = \emptyset$. Pick the first set $M_i$ which has been
removed. It contains at most $3k^{1/3}$ vertices which have at
least $k$ neighboring main sets. Since there are at least
$6k^{1/3}$ vertices in $M_i$, we can pick
$3k^{1/3}$ vertices in $M_i$ which have less than $k$
neighboring main sets and add them to $J$. For each such vertex
added to $J$, remove all the neighboring main sets of it. In this
way, at each step we will increase the size of $J$ by
$3k^{1/3}$ and remove at most $1 + (k-1)\cdot3k^{1/3}$ main
sets. Now pick the first main sets among the remaining ones,
and repeat what we have done
to further increase $J$.

Assume that in the end, there are no remaining sets
(if this is not the case, then we have found our set $S$).
Note that $J$ is an independent set by
construction, and since $s > k^2$, the size of it will be at least
\[ 3k^{1/3} \cdot \frac{k^2}{1+3k^{1/3}\cdot (k-1)} > k. \]
This gives a contradiction and concludes the proof
since the independence number of the graph is
at most $k$.
\end{proof}

From now on we will only consider sets which have indices in $S$.
Let a {\em semi-triangle} be a sequence of three indices $(p, q, r)$
in $S$ which lies in clockwise order on the cycle, and satisfies
either one of the following two conditions (see Figure
\ref{fig:semitriangle}).

\begin{enumerate}[(i)]
  \setlength{\itemsep}{1pt}
  \setlength{\parskip}{0pt}
  \setlength{\parsep}{0pt}
\item Type A : there exists $x_1, x_2 \in I_p', y_1, y_2 \in I_q', z_1, z_2 \in I_r'$ such that $x_1 < x_2, y_1 < y_2, z_1 < z_2$
and $\{x_1, z_1\}, \{x_2, y_1\}, \{y_2, z_2\} \in E(G)$. Moreover, there exists at least one
set $I_i'$ with $i \in S$ in the arc starting at $p$ and ending at $q$ (traverse clockwise).
\item Type B :  there exists $x_1, x_2 \in I_p', y_1, y_2 \in I_q', z_1, z_2 \in I_r'$ such that $x_1 < x_2, y_1 < y_2, z_1 < z_2$
and $\{x_1, y_1\}, \{x_2, z_1\}, \{y_2, z_2\} \in E(G)$.
\end{enumerate}
Note that $(p,q,r)$ being a semi-triangle does not necessarily imply
that $(q,r,p)$ is also a semi-triangle. Semi-triangles are constructed
so that `chords' intersect in a predescribed way. This arrangement
of chords will allow us to find contradicting cycles, once we are given
certain semi-triangles in our graph.
As an instance, one can see that a semi-triangle of Type B contains a cycle
$x_1-y_1-x_2-z_1-y_2-z_2-x_1$, see Figure \ref{fig:semitriangle}.
Recall that each set $I_j'$ lies in a consecutive interval of length
at most $2k$, and thus the length of the cycle is at least $n-6k$.
Moreover, since each set $I_j'$ is defined as the set of odd
vertices in $I_j$, the length of the cycle is at most $n-3$ (it must
miss vertices between $x_1$ and $x_2$, $y_1$ and $y_2$, and $z_1$
and $z_2$). Finally, since all the intervals $I_j$ do not intersect
$W$, the cycle is a contradicting cycle. Therefore we may assume
that no such semi-triangle exists. We will later see that
one can find a contradicting cycle even if Type A
semi-triangles intersect in a certain way.

Next lemma shows that the graph
$G$ contains many semi-triangles of Type A.

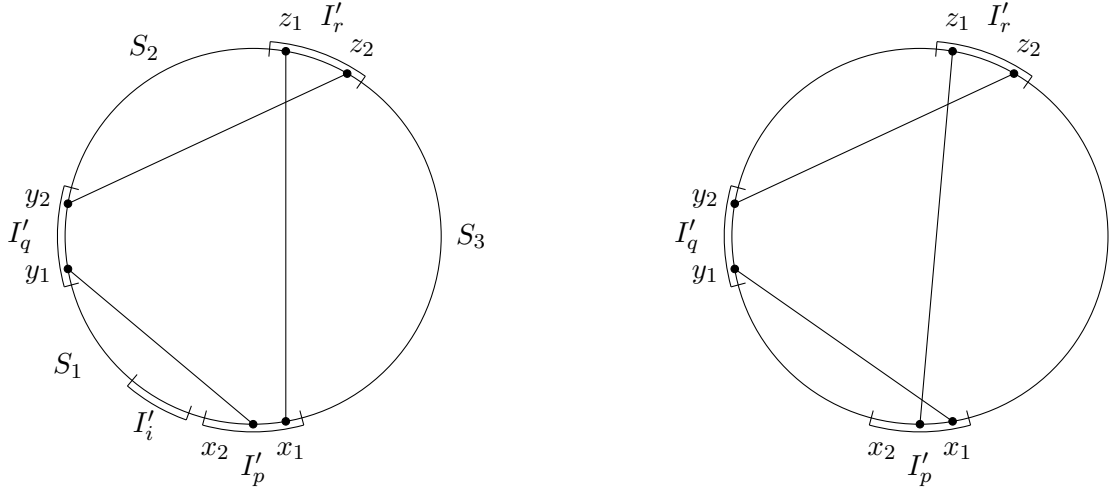
\begin{figure}
  \centering
  \begin{tabular}{ccc}
    %%
%% Semi Triangle - A
%%

\begin{tikzpicture}
% Intervals
  \foreach \a / \t in {
    70/I_r', 180/I_q', 270/I_p'
  }
  {
    \draw (\a-15:24mm) -- (\a-15:26mm);
    \draw (\a+15:24mm) -- (\a+15:26mm);
    \draw (\a-15:26mm) arc (\a-15:\a+15:26mm);
    
    \draw (\a:31mm) node {{$\t$}};
  }
  \draw (230:24mm) -- (230:26mm);
  \draw (250:24mm) -- (250:26mm);
  \draw (230:26mm) arc (230:250:26mm);
  \draw (240:29mm) node {$I_i'$};

% S_1, S_2, S_3
    \draw (120:29mm) node {{\bf $S_2$}};
    \draw (215:30mm) node {{\bf $S_1$}};
    \draw (0:29mm) node {{\bf $S_3$}};

%  Semi triangle
		\draw [fill=black] (60:25mm) circle (0.5mm);
		\draw [fill=black] (80:25mm) circle (0.5mm);
    \draw (60:29mm) node {$z_2$};
    \draw (80:29mm) node {$z_1$};
        		
		\draw [fill=black] (170:25mm) circle (0.5mm);
		\draw [fill=black] (190:25mm) circle (0.5mm);
    \draw (170:29mm) node {$y_2$};
    \draw (190:29mm) node {$y_1$};

		\draw [fill=black] (270:25mm) circle (0.5mm);
		\draw [fill=black] (280:25mm) circle (0.5mm);
    \draw (260:29mm) node {$x_2$};
    \draw (280:29mm) node {$x_1$};
	
		\draw (80:25mm) -- (280:25mm);
		\draw (60:25mm) -- (170:25mm);
		\draw (190:25mm) -- (270:25mm);

% Circle  
  \draw (0:25mm) arc (0:360:25mm);
  
%    \draw [thick, decorate, decoration={snake, amplitude=.2mm, segment length=1mm}] (285:25mm) arc (285:360:25mm);
%    \draw [thick, decorate, decoration={snake, amplitude=.2mm, segment length=1mm}] (0:25mm) arc (0:55:25mm);

\end{tikzpicture} & \hspace{0.5in} & %%
%% Semi Triangle - B
%%

\begin{tikzpicture}
% Intervals
  \foreach \a / \t in {
    70/I_r', 180/I_q', 270/I_p'
  }
  {
    \draw (\a-15:24mm) -- (\a-15:26mm);
    \draw (\a+15:24mm) -- (\a+15:26mm);
    \draw (\a-15:26mm) arc (\a-15:\a+15:26mm);
    
    \draw (\a:31mm) node {{$\t$}};
  }

%  Semi triangle
		\draw [fill=black] (60:25mm) circle (0.5mm);
		\draw [fill=black] (80:25mm) circle (0.5mm);
    \draw (60:29mm) node {$z_2$};
    \draw (80:29mm) node {$z_1$};
		
		\draw [fill=black] (170:25mm) circle (0.5mm);
		\draw [fill=black] (190:25mm) circle (0.5mm);
    \draw (170:29mm) node {$y_2$};
    \draw (190:29mm) node {$y_1$};

		\draw [fill=black] (270:25mm) circle (0.5mm);
		\draw [fill=black] (280:25mm) circle (0.5mm);
    \draw (260:29mm) node {$x_2$};
    \draw (280:29mm) node {$x_1$};
	
		\draw (80:25mm) -- (270:25mm);
		\draw (60:25mm) -- (170:25mm);
		\draw (190:25mm) -- (280:25mm);

% Circle  
  \draw (0:25mm) arc (0:360:25mm);
  
\end{tikzpicture}
  \end{tabular}
  \caption{Semi-triangles of Type A and Type B respectively.}
  \label{fig:semitriangle}
\end{figure}

\begin{LEMMA} \label{lemma_semitriangle}
Let $M_p$ be a fixed main set, and let $S' \subset S$ be a set
of indices such that at least $k^{1/3}$ vertices in $M_p$
have at least $k/3$ neighboring main sets in $S'$. Then
there exists a semi-triangle $(p, q_1, q_2)$ of Type A such that
$q_1,q_2 \in S'$.
\end{LEMMA}
\begin{proof}
Let $M_p$ and $S'$ be given as in the statement.
Among the sets $M_x$ with indices in $S'$, let $M_i$
be the closest one to $M_p$
in the clockwise direction. To make sure that we get a semi-triangle
of Type A, we will remove $i$ from $S'$ and only consider
the set $S'' = S' \setminus \{i\}$.
Thus we will have $k/3-1$ neighboring main sets in $S''$
for each of the given vertices.

Arbitrarily select $k^{1/3}$ vertices in $M_p$ which have at
least $k/3 - 1$ neighboring main sets in $S''$. Since for large $k$ we
have $k^{1/3} \cdot k^{1/3} \le (k/3) - 1$,
we can assign $k^{1/3}$ neighboring main sets to each selected
vertex so that the assigned sets are distinct for different
vertices. Then for a selected vertex $v \in M_p$, let $J_v$ be the
union of the leftover sets $L_x$ corresponding to the $k^{1/3}$
main sets $M_x$ assigned to $v$.
Since each set $L_x$ has size at
least $6k^{1/3}$, by Lemma \ref{lemma_structural}, $J_v$
contains an independent set of size at least $k^{1/3} \cdot
6k^{1/3} - k^{2/3}/2 \ge (11/2)k^{2/3}$. Denote this
independent set by $J_v'$.

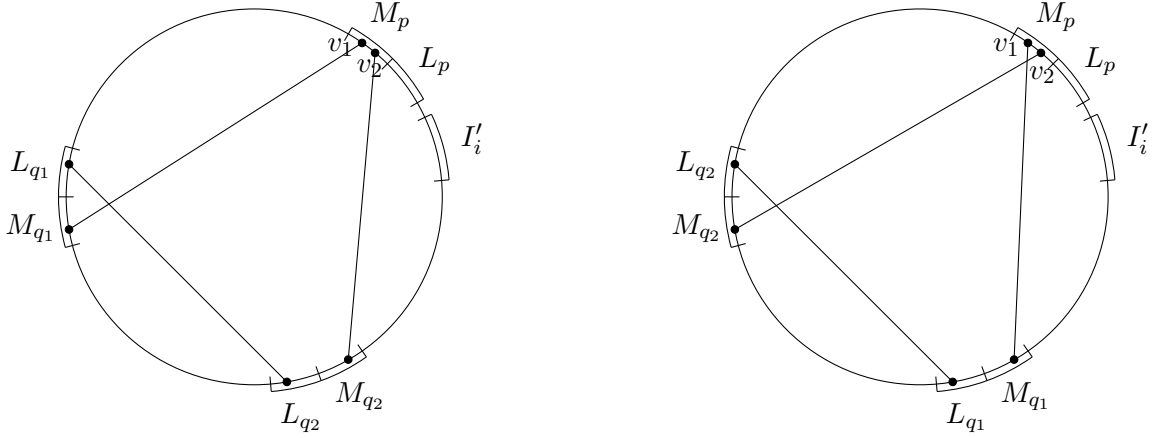
\begin{figure}
  \centering
  \begin{tabular}{ccc}
    %%
%% Semi Triangle - A
%%

\begin{tikzpicture}
% Intervals
  \foreach \a / \t / \s in {
    180/L_{q_1}/M_{q_1}, 290/L_{q_2}/M_{q_2}, 45/L_{p}/M_{p}
  }
  {
    \draw (\a-15:24mm) -- (\a-15:26mm);
    \draw (\a:24mm) -- (\a:26mm);
    \draw (\a+15:24mm) -- (\a+15:26mm);
    
    \draw (\a-15:26mm) arc (\a-15:\a+15:26mm);
    
    \draw (\a-8:30mm) node {{$\t$}};
    \draw (\a+8:30mm) node {{$\s$}};
  }
  
  \draw (5:24mm) -- (5:26mm);
  \draw (25:24mm) -- (25:26mm);
  \draw (5:26mm) arc (5:25:26mm);
	\draw (15:30mm) node {$I_i'$};

%  Semi triangle
		\draw [fill=black] (170:25mm) circle (0.5mm);
		\draw [fill=black] (190:25mm) circle (0.5mm);
%    \draw (60:29mm) node {$z_2$};
%    \draw (80:29mm) node {$z_1$};
        		
		\draw [fill=black] (280:25mm) circle (0.5mm);
		\draw [fill=black] (300:25mm) circle (0.5mm);
%    \draw (170:29mm) node {$y_2$};
%    \draw (190:29mm) node {$y_1$};

		\draw [fill=black] (50:25mm) circle (0.5mm);
		\draw [fill=black] (55:25mm) circle (0.5mm);
    \draw (48:23mm) node {$v_2$};
    \draw (60:23mm) node {$v_1$};
	
		\draw (190:25mm) -- (55:25mm);
		\draw (170:25mm) -- (280:25mm);
		\draw (300:25mm) -- (50:25mm);

% Circle  
  \draw (0:25mm) arc (0:360:25mm);
  
%    \draw [thick, decorate, decoration={snake, amplitude=.2mm, segment length=1mm}] (285:25mm) arc (285:360:25mm);
%    \draw [thick, decorate, decoration={snake, amplitude=.2mm, segment length=1mm}] (0:25mm) arc (0:55:25mm);

\end{tikzpicture} & \hspace{0.5in} & %%
%% Semi Triangle - B
%%

\begin{tikzpicture}
% Intervals
  \foreach \a / \t / \s in {
    180/L_{q_2}/M_{q_2}, 290/L_{q_1}/M_{q_1}, 45/L_{p}/M_{p}
  }
  {
    \draw (\a-15:24mm) -- (\a-15:26mm);
    \draw (\a:24mm) -- (\a:26mm);
    \draw (\a+15:24mm) -- (\a+15:26mm);
    
    \draw (\a-15:26mm) arc (\a-15:\a+15:26mm);
    
    \draw (\a-8:30mm) node {{$\t$}};
    \draw (\a+8:30mm) node {{$\s$}};
  }
  
  \draw (5:24mm) -- (5:26mm);
  \draw (25:24mm) -- (25:26mm);
  \draw (5:26mm) arc (5:25:26mm);
	\draw (15:30mm) node {$I_i'$};

%  Semi triangle
		\draw [fill=black] (170:25mm) circle (0.5mm);
		\draw [fill=black] (190:25mm) circle (0.5mm);
%    \draw (60:29mm) node {$z_2$};
%    \draw (80:29mm) node {$z_1$};
        		
		\draw [fill=black] (280:25mm) circle (0.5mm);
		\draw [fill=black] (300:25mm) circle (0.5mm);
%    \draw (170:29mm) node {$y_2$};
%    \draw (190:29mm) node {$y_1$};

		\draw [fill=black] (50:25mm) circle (0.5mm);
		\draw [fill=black] (55:25mm) circle (0.5mm);
    \draw (45:23mm) node {$v_2$};
    \draw (60:23mm) node {$v_1$};
	
		\draw (190:25mm) -- (50:25mm);
		\draw (170:25mm) -- (280:25mm);
		\draw (300:25mm) -- (55:25mm);

% Circle  
  \draw (0:25mm) arc (0:360:25mm);
  
%    \draw [thick, decorate, decoration={snake, amplitude=.2mm, segment length=1mm}] (285:25mm) arc (285:360:25mm);
%    \draw [thick, decorate, decoration={snake, amplitude=.2mm, segment length=1mm}] (0:25mm) arc (0:55:25mm);

\end{tikzpicture}
  \end{tabular}
  \caption{Constructing semi-triangles, Type A and B, respectively.}
  \label{fig:constructsemitriangle}
\end{figure}

Since the sets $J_v'$ are disjoint for different vertices, we have
$| \bigcup_{v \in M_p} J_v' | \ge (11/2)k^{2/3} \cdot k^{1/3} \ge
k+1$. Therefore, by the restriction on the independence number there
exists an edge between $J_{v_1}'$ and $J_{v_2}'$ for two distinct
vertices $v_1$ and $v_2$ (the edge cannot be within one set $J_v'$
since $J_v'$ is an independent set for all $v$). Let $M_{q_1}$ be
the main set in which the neighborhood of $v_1$ lies in, and
similarly define $M_{q_2}$ so that there exists an edge between
$L_{q_1}$ and $L_{q_2}$. Depending on the relative position of
$M_{q_1}, M_{q_2}$ and $M_p$ on the cycle, the edge $\{v_1,v_2\}$
will give rise to a semi-triangle of Type A or B, see Figure
\ref{fig:constructsemitriangle} (note that the additional condition
for semi-triangle of Type A is satisfied because we removed the
index $i$ in the beginning). Since we know that there does not exist
a semi-triangle of Type B, it should be a semi-triangle of Type A.
\end{proof}

In particular, Lemma \ref{lemma_semitriangle} implies the existence
of a semi-triangle $(p,q,r)$ of Type A. Let the {\em length} of a
Type A semi-triangle be the number of sets $I_i'$ with $i \in S$ in
the arc that starts at $p$ and ends at $q$ (traverse clockwise).
Among all the semi-triangles of Type A consider the one which has
minimum length and let this semi-triangle be $(p,q,r)$. By
definition, every semi-triangle has length at least 1, and thus we
know that there exists an index in $S$ in the arc starting at $p$
and ending at $q$ (traverse clockwise). Let $i \in S$ be such an
index which is closest to $p$ (see Figure
\ref{fig:overlapsemitriangle}).

\begin{figure}
  \centering
  \begin{tabular}{ccc}
    %%
%% Semi Triangle overlap
%%

\begin{tikzpicture}
% Intervals
  \foreach \a / \t in {
    15/I_r', 150/I_q', 270/I_p', 240/I_i', 100/I_j', 50/I_k'
  }
  {
    \draw (\a-10:24mm) -- (\a-10:26mm);
    \draw (\a+10:24mm) -- (\a+10:26mm);
    \draw (\a-10:26mm) arc (\a-10:\a+10:26mm);

    \draw (\a:29mm) node {{$\t$}};
  }

%  Semi triangles
		\draw [fill=black] (10:25mm) circle (0.5mm);
		\draw [fill=black] (20:25mm) circle (0.5mm);
		
		\draw [fill=black] (145:25mm) circle (0.5mm);
		\draw [fill=black] (155:25mm) circle (0.5mm);

		\draw [fill=black] (270:25mm) circle (0.5mm);
		\draw [fill=black] (275:25mm) circle (0.5mm);
	
		\draw (20:25mm) -- (275:25mm);
		\draw (10:25mm) -- (145:25mm);
		\draw (155:25mm) -- (270:25mm);

		\draw [fill=black] (45:25mm) circle (0.5mm);
		\draw [fill=black] (55:25mm) circle (0.5mm);
		
		\draw [fill=black] (95:25mm) circle (0.5mm);
		\draw [fill=black] (105:25mm) circle (0.5mm);

		\draw [fill=black] (240:25mm) circle (0.5mm);
		\draw [fill=black] (245:25mm) circle (0.5mm);
	
		\draw (55:25mm) -- (245:25mm);
		\draw (45:25mm) -- (95:25mm);
		\draw (105:25mm) -- (240:25mm);			

% Circle  
%  \draw (0:25mm) arc (0:360:25mm);

% [decorate, decoration={snake, amplitude=.2mm, segment length=1mm}]
    \draw (20:25mm) arc (20:45:25mm);
    \draw [dotted] (45:25mm) arc (45:55:25mm);
    \draw (55:25mm) arc (55:95:25mm);
    \draw [dotted] (95:25mm) arc (95:105:25mm);    
    \draw (105:25mm) arc (105:145:25mm);
    \draw [dotted] (145:25mm) arc (145:155:25mm);    
    \draw (155:25mm) arc (155:240:25mm);
    \draw [dotted] (240:25mm) arc (240:245:25mm);    
    \draw (245:25mm) arc (245:270:25mm);
    \draw [dotted] (270:25mm) arc (270:275:25mm);    
    \draw (275:25mm) arc (275:360:25mm);
    \draw (0:25mm) arc (0:10:25mm);
		\draw [dotted] (10:25mm) arc (10:20:25mm);

\end{tikzpicture} & \hspace{0.5in} & %%
%% Semi Triangle overlap
%%

\begin{tikzpicture}
% Intervals
  \foreach \a / \t in {
    15/I_r', 150/I_q', 270/I_p', 240/I_i', 340/I_j', 300/I_k'
  }
  {
    \draw (\a-10:24mm) -- (\a-10:26mm);
    \draw (\a+10:24mm) -- (\a+10:26mm);
    \draw (\a-10:26mm) arc (\a-10:\a+10:26mm);

    \draw (\a:29mm) node {{$\t$}};
  }

%  Semi triangles
        \draw [fill=black] (10:25mm) circle (0.5mm);
        \draw [fill=black] (20:25mm) circle (0.5mm);

        \draw [fill=black] (145:25mm) circle (0.5mm);
        \draw [fill=black] (155:25mm) circle (0.5mm);

        \draw [fill=black] (270:25mm) circle (0.5mm);
        \draw [fill=black] (275:25mm) circle (0.5mm);

        \draw (20:25mm) -- (275:25mm);
        \draw (10:25mm) -- (145:25mm);
        \draw (155:25mm) -- (270:25mm);

        \draw [fill=black] (295:25mm) circle (0.5mm);
        \draw [fill=black] (305:25mm) circle (0.5mm);

        \draw [fill=black] (335:25mm) circle (0.5mm);
        \draw [fill=black] (345:25mm) circle (0.5mm);

        \draw [fill=black] (240:25mm) circle (0.5mm);
        \draw [fill=black] (245:25mm) circle (0.5mm);

        \draw (305:25mm) -- (245:25mm);
        \draw (295:25mm) -- (335:25mm);
        \draw (345:25mm) -- (240:25mm);

% Circle
%  \draw (0:25mm) arc (0:360:25mm);

    \draw (20:25mm) arc (20:145:25mm);
    \draw [dotted] (145:25mm) arc (145:155:25mm);
    \draw (155:25mm) arc (155:240:25mm);
    \draw [dotted] (240:25mm) arc (240:245:25mm);
    \draw (245:25mm) arc (245:270:25mm);
    \draw [dotted] (270:25mm) arc (270:275:25mm);
    \draw (275:25mm) arc (275:295:25mm);
    \draw [dotted] (295:25mm) arc (295:305:25mm);
    \draw (305:25mm) arc (305:335:25mm);
    \draw [dotted] (335:25mm) arc (335:345:25mm);
    \draw (345:25mm) arc (345:360:25mm);
    \draw (0:25mm) arc (0:10:25mm);
        \draw [dotted] (10:25mm) arc (10:20:25mm);

\end{tikzpicture}
  \end{tabular}
  \caption{Two overlapping semi-triangles which give a contradicting cycle.}
  \label{fig:overlapsemitriangle}
\end{figure}
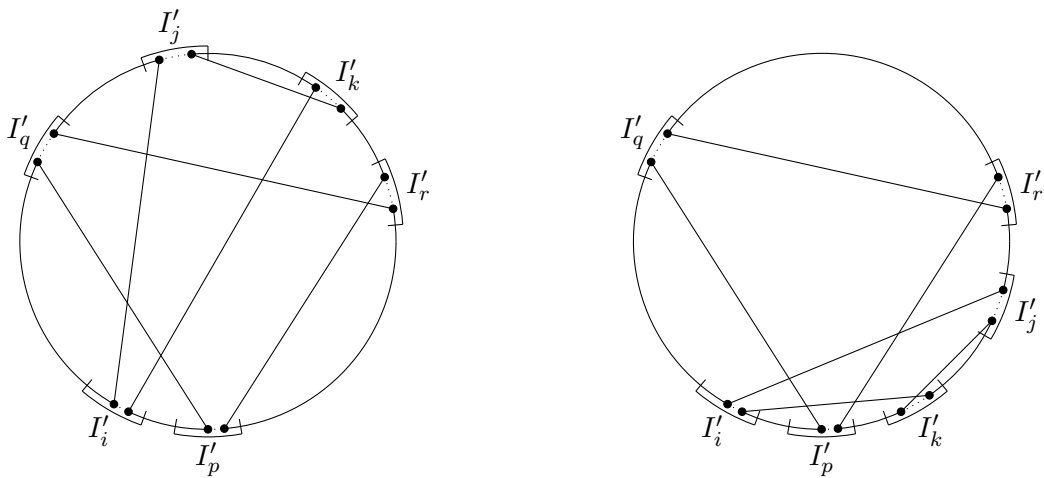

Now consider the set of indices $S_1, S_2, S_3 \subset S$ such that
$S_1$ is the set of indices between $p$ and $q$, $S_2$ is the set of
indices between $q$ and $r$, and $S_3$ is the set of indices between
$r$ and $p$ along the circle, all in clockwise order (see Figure
\ref{fig:semitriangle}). By pigeonhole principle and how we
constructed the indices $S$ in Lemma
\ref{lemma_mindegreereducedgraph}, there exists at least one set out
of $S_1, S_2, S_3$ such that at least $k^{1/3}$ vertices of $M_i$
have at least $k/3$ neighboring main sets inside it.

If this set is $S_1$, then by Lemma \ref{lemma_semitriangle} there
exists a Type A semi-triangle which is completely contained in the
arc between $p$ and $q$, and thus has smaller length than the
semi-triangle $(p,q,r)$. Since this is impossible, we may assume
that the set mentioned above is either $S_2$ or $S_3$. In either
of the cases, by Lemma \ref{lemma_semitriangle} we can find a Type A
semi-triangle $(i,j,k)$ which together with $(p,q,r)$ will give a
contradicting cycle, see Figure \ref{fig:overlapsemitriangle}
(recall that each set $I_i'$ lies in a consecutive interval of
length at most $2k$, and thus the length of this cycle is at least
$n-12k$ and at most $n-6$).
This shows that the assumption we made at the beginning on $G$ not containing
a cycle of length $n-1$ cannot hold. Therefore we have proved Theorem \ref{thm_mainstrong}.

\section{Concluding Remarks}

In this paper we proved that there exists an absolute constant $c$
such that if $G$ is a Hamiltonian graph with $n \ge ck^{7/3}$
vertices and $\alpha(G)\le k$, then $G$ is pancyclic. The main
ingredient of the proof was Theorem \ref{thm_main}, which partially
answers a question of Keevash and Sudakov, and tells us that under
the same condition as above, $G$ contains a cycle of length $n-1$.
It seems very likely that if one can answer Keevash and Sudakov's
question, even for $n = \Omega(k^2)$, then one can also resolve
Erd\H{o}s' question, by using a similar approach to that of Section
\ref{section_thm2} (see Theorem \ref{thm_mainstrong}, which is a
strengthened version of Theorem \ref{thm_main}).

\medskip

\noindent \textbf{Acknowledgement.} We would like to thank Peter Keevash for
stimulating discussions.

\end{document}